%% file: Han_Paper.tex
\documentclass[12pt, reqno]{amsart}
\usepackage{amssymb,amsmath,amsthm,cancel,hyperref,color, enumerate}
\usepackage[pdftex]{graphicx}
\hypersetup{colorlinks=true,linkcolor=blue,citecolor=magenta}
\input{S-preamble.tex}

\begin{document}
\newcounter{itemcounter}
\newcounter{itemcounter2}
\title[A conjecture of Han on 3-cores and modular forms]
{A conjecture of Han on 3-cores and modular forms}
\author{Amanda Clemm}
\address{Department of Mathematics, Emory University, Emory, Atlanta, GA 30322}
\email{aclemm@emory.edu}
\subjclass[2010]{05XX, 11F11}
\keywords{hook length, partition, modular form}
\begin{abstract}
In his study of Nekrasov-Okounkov type formulas on ``partition theoretic" expressions for families of infinite products, Han discovered seemingly unrelated $q$-series that are supported on precisely the same terms as these infinite products. 
In earlier work with Ono, Han proved one instance of this occurrence that exhibited a relation between numbers $a(n)$ that are given in terms of hook lengths of partitions, with numbers $b(n)$ that equal the number of 3-core partitions of $n$. Recently Han revisited the $q$-series with coefficients $a(n)$ and $b(n)$, and numerically found a third $q$-series whose coefficients appear to be supported on the same terms. Here we prove Han's Conjecture about this third series by proving a general theorem about this phenomenon.
\end{abstract}
\maketitle
\noindent

\medskip

\section{Introduction}

In their study of  supersymmetric gauge theory, Nekrasov and Okounkov discovered a striking infinite product identity \cite{nekrasov244seiberg}. This surprising theorem relates the sum over products of partition hook lengths to the powers of Euler products and has been generalized in many ways to give expressions for many infinite product $q$-series. The original identity is given by
\begin{equation*}
F_z(x) := \displaystyle\sum_{\lambda} x^{|\lambda|} \displaystyle\prod_{h \in \mathcal{H}(\lambda)} \left(1-\frac{z}{h^2} \right) = \displaystyle\prod_{n=1}^{\infty} (1-x^n)^{z-1},
\end{equation*}
where the sum is over integer partitions, $|\lambda|$ the integer partitioned by $\lambda$, and $\mathcal{H}(\lambda)$ denotes the multiset of classical hook lengths associated to a partition $\lambda$.  

The Nekrasov-Okounkov formula specializes in the case $z=2$ and $z=4$ to two classical $q$-series identities. The first is a special case of Euler's Pentagonal Number Theorem, and the second gives Jacobi's famous identity for the product $\displaystyle\prod_{n=1}^{\infty}(1-x^n)^3$, \cite{han2009some}. 

\begin{align*}
F_2(x) &= \prod_{n=1}^{\infty}(1-x^n) = \sum_{n=-\infty}^{\infty} (-1)^n x^{\frac{3n^2+n}{2}}, \hspace{.25in} (\mathrm{Euler})
\\ F_4(x) &= \prod_{n=1}^{\infty}(1-x^n)^3 = \sum_{n=0}^{\infty} (-1)^n(2n+1)x^{\frac{n^2+n}{2}}. \hspace{.25in} (\mathrm{Jacobi})
\end{align*}

In \cite{han2008nekrasov}, Han extended the Nekrasov-Okounkov identity to consider the number of $t$-core partitions of $n$. While working on this generalization, Han investigated the nonvanishing of infinite product coefficients. For example, he considers the infinite product, $$\displaystyle\prod_{n \geq 1} \displaystyle\frac{(1-x^{sn})^{t^2/s}}{1-x^n},$$ and conjectures in \cite{han2009some} that the coefficient of $x^n$ is not equal to 0 for $t \geq 5$, $t, s$ positive integers such that $s | t$ and $t \neq 10$. Letting $s=1$ and $t=5$, Han reformulates the famous conjecture of Lehmer that the coefficients of 
\begin{equation*}
x \displaystyle\prod_{n \geq 1}(1-x^n)^{24}=\displaystyle\sum_{n \geq 1} \tau(n) x^n
\end{equation*}
never vanish. 

In \cite{han2009some}, Han formulated a conjecture comparing the nonvanishing of terms of $F_9(x)$ with another infinite Euler product given by 
\begin{align*}
C_3(x) &= \displaystyle\sum_{n=1}^{\infty} b(n)x^n := \displaystyle\prod_{n=1}^{\infty} \displaystyle\frac{(1-x^{3n})^3}{1-x^n}\\
&=1+x+2x^2+2x^4+ \dots +2x^{14}+3x^{16}+2x^{17}+2x^{20} +\dots. 
\tag{1.1}
\label{1.1}
\end{align*} Recall $F_9(x)$ is the series given by
\begin{align*}
F_9(x)&= \displaystyle\sum_{n=1}^{\infty} a(n)x^n := \displaystyle\prod_{n=1}^{\infty} (1-x^n)^8 \\
&=1-8x+20x^2-70x^4+ \dots -520x^{14}+57x^{16}+560x^{17}+182x^{20} +\dots
\tag{1.2}
\label{1.2}
\end{align*} Based on numerical evidence, Han conjectured that the non-zero coefficients of $F_9(x)$ and $C_3(x)$ are supported on the same terms; assuming the notation above, $a(n)=0$ if and only if $b(n)=0$. 

This conjecture is proved in a joint paper with Ono \cite{han2011hook}. In addition to proving the conjecture, Han and Ono proved $a(n)=b(n)=0$ precisely for those non-negative integers $n$ for which $\textrm{ord}_p(3n+1)$ is odd for some prime $p \equiv 2 \pmod{3}$. 

Recently Han discovered another series that appears to be supported on the same terms as $C_3(x)$ and $F_9(x)$. This series is given by,  
\begin{align*}
C(x) &= \displaystyle\sum_{n=1}^{\infty} c(n)x^n := \displaystyle\prod_{n=1}^{\infty} (1-x^n)^2(1-x^{3n})^2 \\
&= 1-2x-x^2+5x^4+ \dots + 8x^{14} -6x^{16}-10x^{17}-x^{20} + \dots
\tag{1.3}
\label{1.3}
\end{align*} Based on numerical evidence, Han conjectured $a(n)=0$ if and only if $b(n)=0$ if and only if $c(n)=0$. 

Here we prove the following general theorem which produces infinitely many modular forms, including those in (\ref{1.2}) and (\ref{1.3}) which are supported precisely on the same terms as (\ref{1.1}). 

It is convenient to normalize (\ref{1.1}) as shown below. 
\begin{align*}
\tag{2.1}
\label{2.1}
\mathcal{B}(z)&= \frac{\eta(9z)^3}{\eta(3z)}=\displaystyle\sum_{n=1}^{\infty} b^*(n)q^n:=\displaystyle\sum_{n=0}^{\infty}b(n)q^{3n+1} \\
&=q+q^4+2q^7+2q^{13}+q^{16}+2q^{19}+q^{25}+2q^{28}+\dots.
\end{align*} 

\begin{thm}
\label{A(n) theorem}
Suppose that $f(z)=\displaystyle\sum_{n=1}^{\infty} A(n)q^n$ is an even weight newform with trivial Nebentypus that has complex multiplication by $\mathbb{Q}(\sqrt{-3})$ and a level of the form $3^s$, where $s \geq 2$. Then the coefficients $A(n)=0$ if and only if $b(n)=0$. More precisely, $A(n)=b^*(n)=0$ for those non-negative integers $n$ for which $\mathrm{ord}_p(n)$ is odd for some prime $p \equiv 2 \pmod{3}$. 
\end{thm}

\begin{rem}
Here we let $q:=e^{2 \pi iz}$ and $\displaystyle\sum_{n=1}^{\infty} A(n)q^n$ is the usual Fourier expansion at infinity. 
\end{rem}

\begin{rem}
Consider the normalized function of $a(n)$ and $c(n)$ given by, 
\begin{align*}
\mathcal{A}(z) &= \displaystyle\sum_{n=1}^{\infty} a^*(n)q^n := \displaystyle\sum_{n=0}^{\infty} a(n)q^{3n+1} \\
&=q-8q^4+20q^7-70q^{13}+64q^{16}+56q^{19}-125q^{25}-160q^{28}+\dots
\tag{2.2}
\label{2.2}
\end{align*}
and 
\begin{align*}
\mathcal{C}(z) &= \displaystyle\sum_{n=1}^{\infty} c^*(n)q^n := \displaystyle\sum_{n=0}^{\infty} c(n)q^{3n+1} \\
&= q - 2q^4 - q^7 + 5q^{13} + 4q^{16} - 7q^{19} - 5q^{25} + 2q^{28} + \dots 
\tag{2.3}
\label{2.3}
\end{align*}
Theorem~\ref{A(n) theorem} implies the original work of Han and Ono. As explained in \cite{han2011hook}, $\mathcal{A}(z)$ is a weight 4 newform with complex multiplication by $\mathbb{Q}(\sqrt{-3})$ with level $9$. Theorem~\ref{A(n) theorem} also implies Han's new conjecture concerning coefficients of $(1.3)$ because $\mathcal{C}(z)$ is the weight 2 complex multiplication form for the elliptic curve with complex multiplication by $\mathbb{Q}(\sqrt{-3})$ given by $y^2+y=x^3-7$ with level $3^3=27$ \cite{martin1997eta}. 
\end{rem}

\begin{rem}
It turns out that more is true about the relationship between the two series in (\ref{1.2}) and (\ref{1.3}). 
If $p \equiv 1 \pmod{3}$ is prime, then we have that $c^*(p)$ divides $a^*(p)$. We will prove this statement in Section 3.1. 
\end{rem}

To prove Theorem~\ref{A(n) theorem}, we make use of the known description of (\ref{1.1}), the generating function for the 3-core partition function, and then generalize the work in \cite{han2011hook} to extend to this situation. 

\section{Preliminaries}
We begin by recalling the exact formula for the coefficients $b^{*}(n)$ of the modular form $\mathcal{B}(z)$ defined below. Recall the Dedekind's eta function, denoted $\eta(z)$, is defined by the infinite product $$\eta(z):=q^{1/24}\prod_{n=1}^{\infty} (1-q^n).$$ The coefficients $b^{*}(n)$ are given by 

\begin{align*}
\tag{2.1}
\label{2.1}
\mathcal{B}(z)&= \frac{\eta(9z)^3}{\eta(3z)}=\displaystyle\sum_{n=1}^{\infty} b^*(n)q^n:=\displaystyle\sum_{n=0}^{\infty}b(n)q^{3n+1} \\
&=q+q^4+2q^7+2q^{13}+q^{16}+2q^{19}+q^{25}+2q^{28}+\dots.
\end{align*} 

\begin{lem}[Lemma 2.5 of \cite{han2011hook}]
Assuming the notation above, we have that 
\label{Lemma 2.4 of H-0}
$$
\mathcal{B}(z)=\displaystyle\sum_{n=1}^{\infty} b^{*}(n)q^n=\displaystyle\sum_{n=0}^{\infty}b(n)q^{3n+1} = \displaystyle\sum_{n=0}^{\infty} \displaystyle\sum_{d|3n+1} \left(\frac{d}{3}\right) q^{3n+1}.
$$\end{lem}

The following lemma describes the nonvanishing conditions for the series (\ref{2.1}) as described in \cite{han2011hook}. 

\begin{lem}
\label{Nonvanishing (1) Lemma}
Assume the notation above. Then $b^*(n)=0$ if and only if $n$ is a non-negative integer for which $\mathrm{ord}_p(n)$ is odd for some prime $p \equiv 2 \pmod{3}$.
\end{lem} 

To prove the original conjecture, Han and Ono recalled the exact formula for the coefficients $a^*(n)$ described in \cite{han2011hook}. The modular form $\mathcal{A}(z)$ is given by 
\begin{equation*}
\tag{2.2}
\mathcal{A}(z)= \eta^8(3z)=\displaystyle\sum_{n=1}^{\infty} a^*(n)q^n:=\displaystyle\sum_{n=0}^{\infty}a(n)q^{3n+1}
\end{equation*} 
where $q:= e^{2 \pi i}$ and $z \in \mathcal{H}$, the upper-half of the complex plane. This normalized series $\mathcal{A}(z)$, such that $a(n) \equiv a^*(3n+1)$, is an example of a special type of modular form. This modular form is in $S_4(\Gamma_0(9))$, the space of weight 4 cusp forms on $\Gamma_0(9)$. $\mathcal{A}(z)$ is a newform with complex multiplication. Using this theory, Han and Ono proved the following theorem. 

\begin{thm}[Theorem 2.1 of \cite{han2011hook}]
\label{Thm 2.1 of H-O}
Assume the notation above. Then the following are true:
\begin{enumerate}
\item If $p \equiv 3$ or $p \equiv 2 \pmod{3}$ is prime, then $a^*(p)=0.$
\item If $p \equiv 1 \pmod{3}$ is prime, then 
\begin{equation*}
a^*(p)=2x^3-18xy^2,
\end{equation*}
where $x$ and $y$ are integers for which $p=x^2+3y^2$ and $x \equiv 1 \pmod{3}$.
\end{enumerate}
\end{thm}
The theorem above shows that $a^*(n)$ satisfies the same nonvanishing conditions demonstrated by $b^{*}(n)$ as noted in Lemma~\ref{Nonvanishing (1) Lemma}, proving the original conjecture. 

\section{Proof of Theorem 1.1}
We now briefly recall the theory of newforms with complex multiplication (see Chapter 12 of \cite{henryk1997topics}
or Section 1.2 of \cite{ken2004web}). Let $D<0$ be the fundamental discriminant of an imaginary quadratic field $K= \mathbb{Q}(\sqrt{D})$. Let $\mathcal{O}_K$ be the ring of integers of $K$. Let $\Lambda$ be a nontrivial ideal in $\mathcal{O}_K$ and $I(\Lambda)$ denote the group of fractional ideals prime to $\Lambda$. Then $\phi$ defines a homomorphism $$\phi:I(\Lambda) \rightarrow \mathbb{C}^{\times}$$ such that for each $\alpha \in K^{\times}$ with $\alpha \equiv 1 \pmod{\Lambda}$, we have $$\phi(\alpha\mathcal{O}_K)=\alpha^{k-1}.$$ Let $\omega_{\phi}$ be the Dirichlet character defined as $$\omega_{\phi}(n):=\phi((n))/n^{k-1}$$ for every integer $n$ coprime to $\Lambda$. Consider the function $\Psi(z)$ defined by $$\Psi(z):=\sum_{\mathfrak{a}} \phi(\mathfrak{a})q^{N(\mathfrak{a})}=\sum_{n=1}^{\infty}a(n)q^n,$$ where the sum is over the integral ideals $\mathfrak{a}$ that are prime to $\Lambda$ and $N(\mathfrak{a})$ is the norm of the ideal $\mathfrak{a}$. This function $\Psi(z)$ is a cusp form in $S_k(\Gamma_0(-D \cdot N(\Lambda)), \left(\frac{-D}{\bullet}\right) \omega_{\phi})$. When $p$ does not divide the level, notice that if $p$ is inert in $K$, then $a(p)=0$ \cite{ken2004web}. \\

The cusp form $\Psi(z)$ is a ``newform'' in the sense of Atkin and Lehner \cite{ken2004web}.  Therefore, $\Psi(z)$ is a normalized cusp form that is an eigenform of all the Hecke operators and all the Atkin-Lehner involutions $|_kW(Q_p)$ for primes $p | N$ and $|_kW(N)$. The following theorem describes the vanishing Hecke eigenvalues when there is a prime $p$ such that $p^2$ divides the level.

\begin{thm}[Theorem 2.27 (3) of \cite{ken2004web}]
\label{2.27 Ono}
Suppose $f(z)= \displaystyle\sum_{n=1}^{\infty} a(n)q^n \in S_k^{\text{new}}(\Gamma_0(N))$ is a newform. If $p$ is a prime for which $p^2|N$, then $a(p)=0$. 
\end{thm} 

This information gives the following nonvanishing conditions on newforms with complex multiplication. 

\begin{lem}
\label{Nonvanishing A(p) Lemma}
Suppose that $f(z)=\displaystyle\sum_{n=1}^{\infty} A(n)q^n$ is an even weight newform with trivial Nebentypus and complex multiplication by $\mathbb{Q}(\sqrt{-3})$ with level of the form $3^s$ where $s\geq 2$. Then $A(p)=0$ if and only if $p=3$ or $p \equiv 2 \pmod{3}$ is prime. 
\end{lem}

\begin{proof}[Proof of Lemma~\ref{Nonvanishing A(p) Lemma}]
The level of $f(z)$ is $3^s$ and therefore $3$ is the only prime that divides the level. Since $k\geq 2$, we know $3^2$ always divides the level, therefore by Theorem~\ref{2.27 Ono} in \cite{ken2004web}, $A(3)=0$. When $p \equiv 2 \pmod{3}$ for $p \neq 3$ prime, $p$ is inert and therefore $A(p)=0$. 
\end{proof}

\begin{cor}
\label{Corollary on A(n)} 
The following are true about $A(n)$. 
\begin{enumerate}
\item If $m$ and $n$ are coprime positive integers, then $$A(mn)=A(m)A(n).$$
\item For every positive integer $t$, we have that $A(3^t)=0$. 
\item If $p \equiv 2 \pmod{3}$ is prime and $t$ is a positive integer, then $A(p^t)=0$ if $t$ is odd and $A(p^t) \neq 0$ if $t$ is even. 
\item If $p \equiv 1 \pmod{3}$, then $A(p^t) \neq 0$. 
\end{enumerate}
\end{cor}

\begin{proof}[Proof of Corollary~\ref{Corollary on A(n)}]
Claim $(1)$ is well known to hold for all normalized Hecke eigenforms. \\
Claim $(2)$ follows as $A(3)=0$. \\
To prove Claim $(3)$, observe that every newform is a Hecke eigenform. Moreover, since $A(1)=1$, the Hecke eigenvalue of $T(p)$ is $A(p)$.
Therefore, for every integer $n$ and prime $p \neq 3$, we have that $$A(p)A(n)= A(pn)+p^{k-1}A(n/p).$$ The left hand side of the equation is the statement that $A(p)$ is the Hecke eigenvalue. The right hand side of the equation is the action of the Hecke operator $T(p)$. Let $n=p^t$ and $p\equiv 2 \pmod{3}$ be prime. Since $A(p)=0$ for $p \equiv 2 \pmod{3}$, this equation becomes $$0= A(p^{t+1})+p^{k-1}A(p^{t-1}).$$
Claim $(3)$ follows from induction as $A(1)=1$ and $A(p)=0$. \\
To prove Claim $(4)$, let $p$ be a prime such that $p \equiv 1 \pmod{3}$. Suppose that $A(p)=0$; this implies that $\alpha$ is totally imaginary. But then $$p = (\beta \sqrt{-3)})(-\beta\sqrt{-3})=3\beta^2,$$
which is false. Claim $(4)$ then follows by induction. 
\end{proof}

\begin{proof}[Proof of Theorem~\ref{A(n) theorem}]
The theorem follows by combining Lemma~\ref{Nonvanishing (1) Lemma}, Corollary~\ref{Corollary on A(n)} and Lemma~\ref{Nonvanishing A(p) Lemma}. 
\end{proof}

\subsection{The series in (\ref{1.3})}
We normalize the function $c(n)$ using the following series, 
\begin{equation*}
\mathcal{C}(z)= \displaystyle\sum_{n=1}^{\infty} c^*(n) q^n := \displaystyle\sum_{n=0}^{\infty} c(n) q^{3n+1}.
\end{equation*}
The series $\mathcal{C}(z)$ is a modular form given by 
\begin{equation*}
\mathcal{C}(z)= \eta^2(3z) \eta^2(9z) = \displaystyle\sum_{n=1}^{\infty} c^*(n) q^n.
\end{equation*}

In \cite{martin1997eta}, Martin and Ono give a complete description of all weight 2 newforms that are products and quotients of the Dedekind eta-function. The descriptions in \cite{martin1997eta} include formulas for the $p^{th}$ coefficients. Since these coefficients are Hecke multiplicative, it suffices to give the formula for only $p$ prime. Specifically, for $\mathcal{C}(z)$, we have the following theorem. 
\begin{thm}[Theorem 2 in \cite{martin1997eta}]
Assuming the notation above, the following are true. 
\begin{enumerate}
\item If $p \equiv 2 \pmod{3}$, then $c^*(p)=0$. 
\item If $p \equiv 1 \pmod{3}$, then $c^*(p)=2m+n$ where $p=m^2+mn+n^2$ and $m \equiv 1 \pmod{3}$ and $n \equiv 0 \pmod{3}$.
\end{enumerate} 
\end{thm}

Recall \ref{Thm 2.1 of H-O} from \cite{han2011hook}, gave the following conditions on the coefficients of $\mathcal{A}(z)$:
If $p \equiv 1 \pmod{3}$ is prime, then 
\begin{equation*}
a^*(p)=2x^3-18xy^2,
\end{equation*}
where $x$ and $y$ are integers for which $p=x^2+3y^2$ and $x \equiv 1 \pmod{3}$.

Here we show that $c^*(p) | a^*(p)$ for primes $p \equiv 1 \pmod{3}$ and $n$ even:
\begin{align*}
p &=m^2+mn+n^2 \\
&=\left(m + \frac{n}{2}\right)^2+3\left(\frac{n}{2}\right)^2 \\
&=x^2+3y^2.
\end{align*}

Let $x=\left(m + \frac{n}{2}\right)$ and $y=\displaystyle\frac{n}{2}$. Then 
\begin{align*}
a^*(p)&=2x^3-18xy^2 \\
&=2\left(m + \frac{n}{2}\right)^3-18\left(m + \frac{n}{2}\right)\left(\frac{n}{2}\right)^2 \\
&=(2m+n)(m+2n)(m-n) \\
&=c^*(p)(m+2n)(m-n).
\end{align*}

Since $m \equiv 1 \pmod{3}$ and $n \equiv 0 \pmod{3}$, we have $a^*(p)\equiv c^*(p) \pmod{3}$ and as mentioned in a remark, $c^*(p) | a^*(p)$. 

\bibliographystyle{plain}
\bibliography{Hanrefs.bib}

\end{document}

%% file: S-preamble.tex
\newtheorem{thm}{Theorem}[section]
\newtheorem{lem}[thm]{Lemma}

\newtheorem*{rem}{Remark}

\newtheorem{cor}[thm]{Corollary}